\documentclass[final,3p,times]{elsarticle}

\usepackage[T1]{fontenc}
\usepackage[utf8]{inputenc}
\usepackage{lmodern}
\usepackage{microtype}
\usepackage{amsmath,amssymb,mathtools,amsthm}
\usepackage{enumitem}
\usepackage{xcolor}
\usepackage[colorlinks=true,linkcolor=blue!45!black,citecolor=blue!45!black,urlcolor=blue!45!black]{hyperref}

\biboptions{sort&compress}

\newtheorem{theorem}{Theorem}[section]
\newtheorem{proposition}[theorem]{Proposition}
\newtheorem{lemma}[theorem]{Lemma}
\newtheorem{corollary}[theorem]{Corollary}
\theoremstyle{definition}

\newtheorem{remark}[theorem]{Remark}

\newcommand{\R}{\mathbb R}
\newcommand{\id}{\operatorname{id}}
\newcommand{\Int}{\operatorname{int}}
\newcommand{\relint}{\operatorname{relint}}
\newcommand{\cone}{\operatorname{cone}}
\newcommand{\conv}{\operatorname{conv}}
\newcommand{\supp}{\operatorname{supp}}
\newcommand{\Min}{\operatorname{Min}}

\newcommand{\Dia}{\Diamond}
\newcommand{\Vone}{\mathcal V_1}
\newcommand{\Vsub}{\mathcal V_{\leq 1}}
\newcommand{\Vext}{\mathcal V}

\newcommand{\dd}{\,d}
\newcommand{\one}{\mathbf 1}

\journal{Preprint}

\begin{document}

\begin{frontmatter}

\title{Characterizing finite posets whose probabilistic powerdomains are RB-domains}

\author[addr1]{Yuxu Chen}
\address[addr1]{School of Mathematics, Sichuan University}
\ead{chenyuxu@scu.edu.cn}

\author[addr1]{Hui Kou}
\ead{kouhuixu@scu.edu.cn}

\author[addr1]{Zhenchao Lyu}
\ead{zhenchaolyu@scu.edu.cn}

\begin{abstract}
We classify the finite posets whose probabilistic powerdomain is an RB-domain.  For a finite nonempty poset \(P\), let \(\Vone(P)\) be the  probability powerdomain of $P$, which is the probability simplex ordered by the stochastic order.  We prove that \(\Vone(P)\) is an RB-domain if and only if \(P\) has a least element and the undirected Hasse graph of \(P\) is a tree.  Consequently, the  probabilistic powerdomain does not preserve RB-domains; the four-point diamond gives a finite counterexample.  The proof separates two obstructions.  First, if \(P\) has no least element, then the face of probability measures supported on the minimal points must be fixed pointwise by every deflation below the identity.  
Secondly, once a least element exists, the Hasse graph is connected, and a cycle in it makes the local stochastic cone non-simplicial.  A Euclidean finite-step cone argument then rules out the finite-valued monotone approximations supplied by the RB property.
\end{abstract}


\end{frontmatter}

\section{Introduction}

The probabilistic powerdomain is a basic construction in domain-theoretic semantics for probabilistic computation, going back to Jones and Plotkin \cite{JonesPlotkin1989}.  A natural structural question is whether this construction preserves the standard approximation classes of domains.  Jung and Tix showed that a previously claimed proof of preservation of RB-domains was not valid and isolated finite posets as basic test cases \cite{JungTix1998}.  They obtained a positive result for finite tree domains, while the behavior for arbitrary finite posets remained open.  Goubault-Larrecq later emphasized the same difficulty for probabilistic and subprobabilistic powerdomains and singled out the four-point diamond as an important finite example \cite{GoubaultLarrecq2012}.

The purpose of this paper is to classify all finite posets whose  probabilistic powerdomain is an RB-domain.  
The main ingredient is a finite-dimensional cone obstruction: for a non-simplicial proper cone, finite-valued monotone maps cannot locally approximate the identity.  The proof uses elementary convex separation, Lipschitz epigraphs, Rademacher's theorem, Fubini's theorem, and integration by parts.

In this paper, we give the exact answer for  probabilisic powerdomain on finite posets.  If \(P\) is finite and nonempty, then a probability valuation is simply a vector
\[
        p=(p_x)_{x\in P}\in\R^P_{\ge0},
        \qquad
        \sum_{x\in P}p_x=1.
\]
The order is the stochastic order:
\[
        p\le q
        \quad\Longleftrightarrow\quad
        p(U)\le q(U)
        \quad\text{for every upper set }U\subseteq P,
\]
where \(p(U)=\sum_{x\in U}p_x\).  This ordered dcpo is the probabilistic powerdomain of $P$, denoted by \(\Vone(P)\).

The main theorem is the following. For a finite poset:
\[
\boxed{
        \Vone(P)\text{ is an RB-domain}
        \quad\Longleftrightarrow\quad
        P\text{ is a pointed tree }}
\]
Thus the  probabilistic powerdomain is RB precisely on finite tree domains.  Here a finite tree domain is meant in the standard sense used in the finite-tree theorem of Jung--Tix and in Goubault-Larrecq's Lemma~6.8: a finite pointed poset whose principal ideals, or downward closures, are chains \cite[Theorem~13]{JungTix1998}; see also \cite[Lemma~6.8]{GoubaultLarrecq2012}.  For finite posets this is equivalent to having a least element and a tree as undirected Hasse graph, as recalled below.

The proof is organized as two obstructions followed by the known positive tree case.  The first obstruction is purely order-theoretic.  If \(P\) has no least element, then the minimal points form a set \(M\) with at least two elements.  The face \(\Delta_M\) of probabilities supported on \(M\) consists entirely of minimal elements of \(\Vone(P)\).  Therefore every deflation \(r\leq\id\) fixes \(\Delta_M\) pointwise, contradicting finite range.

After this, we may assume that \(P\) has a least element.  Then the Hasse graph is connected, and the stochastic order on \(\Delta_P\) has a single local order cone
\[
        K_P=\cone\{\delta_y-\delta_x:x\prec y\},
\]
the cone of nonnegative upward flows on the Hasse diagram.  If the Hasse graph has a cycle, then this Hasse-flow cone is non-simplicial.  The RB approximation property would produce finite-valued \(K_P\)-monotone maps on compact blocks of \(\Delta_P\) converging uniformly to the identity.  A Euclidean finite-step obstruction shows that this is impossible for non-simplicial closed pointed cones.

Since this paper deals with general finite posets, the notation is somewhat heavier than in the four-point example.  For orientation, one may also compare the special diamond-lattice calculation, where the same obstruction can be written out explicitly.

Section \ref{sec:valuations} recalls the stochastic order and the RB finite-deflation property.  Section \ref{sec:least} proves the least-element obstruction.  Section \ref{sec:hasse} identifies the local order cone with the Hasse-flow cone and derives the local approximation supplied by RB.  Section \ref{sec:euclidean} proves the Euclidean finite-step obstruction.  Section \ref{sec:classification} combines these ingredients with the Jung--Tix finite-tree theorem.

\section{Finite probability valuations and RB approximants}\label{sec:valuations}

\subsection{Domain-theoretic notation}

We use standard terminology from domain theory; see \cite{AbramskyJung1994,Gierz2003}.  A subset \(A\) of a poset is \emph{directed} if it is nonempty and every two elements of \(A\) have an upper bound in \(A\).  A \emph{dcpo} is a poset in which every directed subset has a supremum.  A map between dcpos is \emph{Scott-continuous} if it is monotone and preserves directed suprema.  For elements \(x,y\) in a dcpo, one writes \(x\ll y\) if, whenever \(A\) is directed and \(y\le \sup A\), there is \(a\in A\) such that \(x\le a\).

A \emph{deflation} on a dcpo \(D\) is a Scott-continuous map \(r:D\to D\) with finite image and \(r\leq\id_D\).  An RB-domain may be characterized internally as a pointed dcpo carrying a directed family \((r_i)_{i\in I}\) of deflations with pointwise supremum \(\id_D\); see Jung \cite[Theorem~4.1]{Jung1989} or Jung--Tix \cite{JungTix1998}.  We shall use only this finite-deflation approximation property.

\subsection{Finite-dimensional and convex-analytic notation}\label{subsec:finite-dimensional-notation}

All finite-dimensional vector spaces in the paper are real.  We use standard
facts from finite-dimensional convex analysis, including separation of closed
convex sets, relative interiors, dual cones, and the bipolar theorem; see, for
example, Rockafellar \cite{Rockafellar1970}.

In the Euclidean finite-step part, Euclidean vectors are typeset in bold, for
example \(\mathbf x,\mathbf y,\mathbf v,\mathbf h\).  In Section~\ref{sec:euclidean}, the distinguished cone direction and dual vector are written \(\mathbf e\) and \(\boldsymbol\eta\), and transverse variables are written \(\mathbf z,\mathbf w\).  Poset elements and probability valuations retain the conventional notation used in finite-poset theory and are not
bolded.  Thus, in expressions such as \(\delta_y-\delta_x\), the symbols
\(x,y\in P\) are elements of the finite poset, while \(\delta_x,\delta_y\in
\R^P\) are basis vectors.  We write \(\mathbf x\cdot\mathbf y\) for the
Euclidean inner product and \(\|\mathbf x\|\) for the associated norm.  We use
\(\operatorname{End}(E)\) for the vector space of linear maps \(E\to E\); after
choosing Euclidean coordinates, these are matrices.  The identity matrix on
\(\R^d\) is denoted by \(I_d\).  Vectors in \(\R^d\) are regarded as column
vectors.  For \(\mathbf v,\mathbf y\in\R^d\), the product
\(\mathbf v\mathbf y^T\) denotes the rank-one outer-product matrix
\[
        (\mathbf v\mathbf y^T)\mathbf h
        =\mathbf v\,(\mathbf y\cdot \mathbf h),
        \qquad \mathbf h\in\R^d .
\]

For a subset \(A\) of a Euclidean space, \(\Int A\) denotes ordinary interior
in the ambient space, and \(\relint A\) denotes relative interior in the
affine hull of \(A\).  We write \(L\Subset B\) to mean that \(L\) is
compactly contained in \(B\); in the situations below this means that the closure
of \(L\), taken in the relevant Euclidean or affine ambient space, is compact
and contained in \(B\).  If \(S\) is compact and convex, an
\emph{extreme point} of \(S\) is a point that is not a nontrivial convex
combination of two other points of \(S\).

For a subset \(S\) of a vector space, \(\conv(S)\) denotes the convex hull and
\(\cone(S)\) denotes the set of all finite nonnegative linear combinations of
elements of \(S\).  A \emph{cone}
\(K\) is a set closed under addition and multiplication by nonnegative
scalars.  It is \emph{pointed} if \(K\cap(-K)=\{0\}\), and
\emph{full-dimensional} if its ordinary interior in the ambient space is
nonempty.  A cone is \emph{proper} if it is closed, convex, pointed, and
full-dimensional.  A cone is \emph{polyhedral} if it is generated by finitely
many vectors, equivalently, in finite dimension, if it is the intersection of
finitely many closed half-spaces.  A ray \(\R_{\ge0}\mathbf v\), with
\(\mathbf v\ne0\), is an \emph{extreme ray} of a cone \(C\) if
\(\mathbf v=\mathbf a+\mathbf b\) with \(\mathbf a,\mathbf b\in C\) implies
\(\mathbf a,\mathbf b\in\R_{\ge0}\mathbf v\).  If a pointed cone spans a
\(d\)-dimensional vector space, it is \emph{simplicial} if it is generated by
\(d\) linearly independent rays; for pointed polyhedral cones this is
equivalent to having exactly \(d\) extreme rays.

For a cone \(K\) in a Euclidean space \(E\), the Euclidean dual cone is
\[
        K^*=\{\mathbf y\in E:\mathbf y\cdot \mathbf x\ge0
        \text{ for every }\mathbf x\in K\}.
\]
The cone order associated with \(K\) is
\[
        \mathbf u\le_K \mathbf v\quad\Longleftrightarrow\quad
        \mathbf v-\mathbf u\in K.
\]
If \(K\) is pointed, this is a partial order.  A map \(Q:X\to E\), where
\(X\subseteq E\), is \(K\)-monotone if \(\mathbf x\le_K \mathbf y\) in \(X\)
implies \(Q(\mathbf x)\le_K Q(\mathbf y)\).

If \(O\subseteq\R^d\) is open, \(C_c^\infty(O)\) denotes the space of smooth
real-valued functions on \(\R^d\) whose supports are compact subsets of
\(O\).  The support of a function \(g\), denoted \(\supp g\), is the closure of
\(\{\mathbf x:g(\mathbf x)\ne0\}\).  Lebesgue measure on \(\R^d\) is denoted in
integrals by \(\dd x\).  Thus \(\int f(\mathbf x)\dd x\) means integration with respect
to \(d\)-dimensional Lebesgue measure; the symbol \(\dd x\) is a measure
notation, not a derivative.  If \(W\) is a Euclidean subspace, Lebesgue measure
on \(W\) is denoted by the corresponding variable, for example \(\dd z\).  In
Section~\ref{sec:euclidean} we sometimes normalize the measure \(\dd z\) on a
chosen hyperplane \(W\) so that a specified linear coordinate change has
Jacobian one.

The gradient notation is used only in the Euclidean integration argument of
Section~\ref{sec:euclidean}.  For \(\psi\in C_c^\infty(O)\), \(\nabla\psi(\mathbf x)\)
denotes the Euclidean gradient, characterized by
\[
        \left.\frac{d}{ds}\right|_{s=0}\psi(\mathbf x+s\mathbf h)
        =\nabla\psi(\mathbf x)\cdot \mathbf h,
        \qquad \mathbf h\in\R^d .
\]
If \(W\subseteq\R^d\) is a Euclidean subspace, \(V\subseteq W\) is open, and
\(g:V\to\R\) is differentiable at \(\mathbf z\in V\), then \(\nabla_W g(\mathbf z)\in W\) is the
unique vector satisfying
\[
        Dg_{\mathbf z}(\mathbf w)=\nabla_W g(\mathbf z)\cdot \mathbf w,
        \qquad \mathbf w\in W.
\]
Here the dot product is the Euclidean inner product restricted to \(W\).  Thus
\(\nabla_W g(\mathbf z)\) records the direction of steepest increase among directions
lying in the subspace \(W\); it has no component normal to \(W\).

\subsection{Finite probability valuations}

Let \(P\) be a finite nonempty poset.  An upper set is a subset \(U\subseteq P\) such that \(x\in U\) and \(x\le y\) imply \(y\in U\).  We write
\[
        \Delta_P=
        \left\{p\in\R^P_{\ge0}:\sum_{x\in P}p_x=1\right\}
\]
for the probability simplex on \(P\).  The stochastic order on \(\Delta_P\) is
\[
        p\le q
        \quad\Longleftrightarrow\quad
        p(U)\le q(U)
        \quad\text{for every upper set }U\subseteq P.
\]
We denote the ordered set \((\Delta_P,\leq)\) by \(\Vone(P)\).  The next lemma is the standard finite-space fact that probability valuations form a dcpo and that directed suprema are computed on upper-set evaluations; we include the proof because the exact formula is used later.

\begin{lemma}[Directed suprema on a finite poset]\label{lem:finite-directed-sup}
Let \(A\subseteq\Delta_P\) be directed.  Then \(A\) has a supremum in \(\Delta_P\), and for every upper set \(U\subseteq P\),
\[
        (\sup A)(U)=\sup_{p\in A}p(U).
\]
\end{lemma}

\begin{proof}
Since \(P\) is finite, \(\Delta_P\) is compact in the Euclidean topology of \(\R^P\).  Regard \(A\) as a net indexed by itself.  Choose a convergent subnet \((p_j)\) with limit \(p\in\Delta_P\).

We first show that \(p\) is an upper bound of \(A\).  Fix \(p_0\in A\).  The tail \(A_{p_0}=\{a\in A:p_0\le a\}\) is cofinal in \(A\), so the subnet is eventually in this tail.  Hence, for every upper set \(U\), eventually
\[
        p_0(U)\le p_j(U).
\]
The evaluation map \(\eta\mapsto\eta(U)\) is continuous, so passing to the limit gives \(p_0(U)\le p(U)\) for every upper \(U\).  Thus \(p_0\le p\), and \(p\) is an upper bound.

If \(q\) is any upper bound of \(A\), then \(p_j\le q\) for every \(j\).  Passing to the limit in all upper-set evaluations gives \(p\le q\).  Hence \(p=\sup A\).

Finally, for an upper set \(U\), the inequality \(\sup_{a\in A}a(U)\le p(U)\) follows from \(a\le p\).  The reverse inequality follows from \(p(U)=\lim_j p_j(U)\) and \(p_j(U)\le \sup_{a\in A}a(U)\).  This proves the formula.  Since the stochastic order is defined by upper-set evaluations, the supremum is unique.
\end{proof}

\subsection{Finite-poset notation}

If \(P\) is finite and \(x<y\), we say that \(y\) \emph{covers} \(x\), and write \(x\prec y\), if there is no \(z\in P\) with \(x<z<y\).  The Hasse diagram has vertex set \(P\) and directed edges \(x\to y\) for cover relations \(x\prec y\).  The \emph{undirected Hasse graph} is obtained by forgetting these orientations.  We use standard graph terminology: a graph is a \emph{tree} if it is connected and has no cycle \cite{Diestel2017}.  A saturated chain is a chain \(x_0<\cdots<x_n\) such that \(x_{i-1}\prec x_i\) for all \(i\).

In the finite-tree theorem of Jung--Tix and in Goubault-Larrecq's Lemma~6.8,
a finite tree is a finite pointed poset whose principal ideals \(\downarrow x\)
are chains \cite[Theorem~13]{JungTix1998}; see also \cite[Lemma~6.8]{GoubaultLarrecq2012}.
Lemma~\ref{lem:finite-tree-equivalence} spells out the equivalence with the
Hasse-graph condition used in this paper. Since this is a standard elementary equivalence for finite posets, we omit the proof.

\begin{lemma}[Finite tree domains and Hasse trees]\label{lem:finite-tree-equivalence}
For a finite nonempty poset \(P\), the following are equivalent:
\begin{enumerate}[label=\textup{(\roman*)}]
\item \(P\) has a least element and its undirected Hasse graph is a tree;
\item \(P\) has a least element and every principal ideal \(\downarrow x\) is a chain.
\end{enumerate}
\end{lemma}

For a vector \(h\in\R^P\) and a subset \(A\subseteq P\), we write
\[
        h(A)=\sum_{x\in A}h_x.
\]
The vector \(\one_A\) is the indicator of \(A\), and \(\delta_x=\one_{\{x\}}\) is the unit vector at \(x\).  We write
\[
        \uparrow x=\{y\in P:x\le y\}
\]
for the principal upper set generated by \(x\).

\section{The least-element obstruction}\label{sec:least}

The first obstruction is independent of the Hasse graph.  It shows that the  probabilistic powerdomain can be RB only when the underlying finite poset has a least element.

\begin{proposition}[No finite deflation without a least element]\label{prop:least-element-obstruction}
Let \(P\) be a finite nonempty poset.  If \(P\) has no least element, then \(\Vone(P)\) admits no finite-range deflation \(r\leq\id\).  Consequently, if \(\Vone(P)\) is an RB-domain, then \(P\) has a least element.
\end{proposition}

\begin{proof}
Suppose that \(P\) has no least element.  Since \(P\) is finite, the set
\[
        M=\Min(P)
\]
of minimal elements is nonempty.  It has at least two elements: if \(M=\{m\}\), then every element of \(P\) lies above some minimal element, hence above \(m\), so \(m\) would be least.  Let
\[
        \Delta_M=
        \{\mu\in\Delta_P:\supp(\mu)\subseteq M\}
\]
be the face of probabilities supported on \(M\).  This is a positive-dimensional simplex, hence infinite.

We claim that every \(\mu\in\Delta_M\) is minimal in \(\Vone(P)\).  Let \(\nu\leq\mu\).  If \(x\notin M\), then \(\uparrow x\) contains no minimal element.  Hence
\[
        \mu(\uparrow x)=0.
\]
Since \(\nu\leq\mu\), we get \(\nu(\uparrow x)=0\), and therefore \(\nu_x=0\).  Thus \(\nu\) is supported on \(M\).

For \(m\in M\), the upper set \(\uparrow m\) meets \(M\) only at \(m\).  Hence
\[
        \nu_m=\nu(\uparrow m)
        \leq
        \mu(\uparrow m)=\mu_m.
\]
Summing over \(m\in M\) gives
\[
        1=\sum_{m\in M}\nu_m
        \leq
        \sum_{m\in M}\mu_m=1.
\]
Thus equality holds termwise, and \(\nu_m=\mu_m\) for all \(m\in M\).  Therefore \(\nu=\mu\), proving minimality.

If \(r\leq\id\) is a deflation and \(\mu\in\Delta_M\), then \(r(\mu)\leq\mu\).  By minimality, \(r(\mu)=\mu\).  Thus \(r\) fixes \(\Delta_M\) pointwise, so its image contains the infinite set \(\Delta_M\).  This contradicts the finite-image condition.  Hence no such finite-range deflation exists when \(P\) has no least element.
\end{proof}

\begin{corollary}\label{cor:least-connected}
If \(P\) has a least element, then the undirected Hasse graph of \(P\) is connected.
\end{corollary}

\begin{proof}
Let \(\bot\) be the least element.  For every \(x\in P\), there is a saturated chain from \(\bot\) to \(x\).  Hence every vertex is connected to \(\bot\) in the undirected Hasse graph.
\end{proof}

\section{Hasse-flow cones and local RB approximation}\label{sec:hasse}

From now on in this section, \(P\) is a connected finite poset.  This is the only case needed after Proposition \ref{prop:least-element-obstruction}, since the existence of a least element implies connectedness.

\subsection{The local stochastic cone}

Set
\[
        H_P=\left\{h\in\R^P:\sum_{x\in P}h_x=0\right\},
\]
the affine tangent space of the simplex \(\Delta_P\).  Define
\[
        K_P=
        \{h\in H_P:h(U)\ge0\text{ for every upper set }U\subseteq P\}.
\]
Then for \(p,q\in\Delta_P\),
\[
        p\le q
        \quad\Longleftrightarrow\quad
        q-p\in K_P.
\]
Thus the stochastic order is, locally and globally on the simplex, the cone order induced by \(K_P\).

We use the convex-analytic terminology fixed in Subsection~\ref{subsec:finite-dimensional-notation}.

\begin{lemma}[Hasse-flow formula]\label{lem:hasse-flow}
If the undirected Hasse graph of \(P\) is connected, then
\[
        K_P=\cone\{\delta_y-\delta_x:x\prec y\}.
\]
Thus \(K_P\) is the cone of nonnegative upward flows on the Hasse diagram.
Moreover, \(K_P\) is closed and pointed, \(K_P-K_P=H_P\), and \(K_P\) has
nonempty relative interior in \(H_P\).
\end{lemma}

\begin{proof}
For a cover edge \(x\prec y\), the vector
\[
        \delta_y-\delta_x
\]
is its incidence vector: it has coordinate \(-1\) at \(x\), coordinate \(+1\)
at \(y\), and coordinate \(0\) elsewhere.  Its coordinate sum is \(0\), so it
belongs to \(H_P\).  Put
\[
        C_P=\cone\{\delta_y-\delta_x:x\prec y\}\subseteq H_P.
\]
We first prove \(C_P=K_P\) by identifying the dual cone \(C_P^*\subseteq
H_P^*\).

A linear functional on \(H_P\) can be represented by a function
\(\varphi:P\to\R\) via
\[
        h\longmapsto \sum_{z\in P}\varphi(z)h_z .
\]
This representative is unique only modulo constants.  Indeed, if \(c\in\R\),
then for every \(h\in H_P\),
\[
        \sum_{z\in P}c h_z=c\sum_{z\in P}h_z=0.
\]
Thus adding a constant function to \(\varphi\) does not change the induced
functional on \(H_P\).  We write \(\one_A\) for the indicator function of a
subset \(A\subseteq P\); in particular, \(\one_P\) is the function constantly
equal to \(1\) on \(P\).

The functional represented by \(\varphi\) is nonnegative on \(C_P\) if and only
if it is nonnegative on each generator \(\delta_y-\delta_x\), \(x\prec y\).  But
\[
        \varphi(\delta_y-\delta_x)=\varphi(y)-\varphi(x).
\]
Hence nonnegativity on \(C_P\) is equivalent to
\[
        \varphi(x)\le \varphi(y)\qquad(x\prec y).
\]
For a finite poset, monotonicity on all cover relations is equivalent to
monotonicity on the whole order: if \(x\le y\), then there is a saturated chain
\[
        x=x_0\prec x_1\prec\cdots\prec x_n=y,
\]
and the cover inequalities give
\[
        \varphi(x)=\varphi(x_0)\le\varphi(x_1)\le\cdots\le\varphi(x_n)=\varphi(y).
\]
Thus the functionals in \(C_P^*\) are represented, modulo constants, by
order-preserving functions \(\varphi:P\to\R\).

Let \(\varphi\) be order-preserving, and let its distinct values be
\[
        c_0<c_1<\cdots<c_m.
\]
For \(1\le k\le m\), set
\[
        U_k=\{x\in P:\varphi(x)\ge c_k\}.
\]
Each \(U_k\) is an upper set: if \(x\in U_k\) and \(x\le y\), then
\(\varphi(y)\ge\varphi(x)\ge c_k\), so \(y\in U_k\).  The usual layer
identity gives
\[
        \varphi
        =c_0\one_P+
        \sum_{k=1}^m(c_k-c_{k-1})\one_{U_k}.
\]
Indeed, if \(\varphi(x)=c_j\), then \(x\in U_k\) exactly for \(1\le k\le j\),
and the right-hand side at \(x\) is
\[
        c_0+\sum_{k=1}^j(c_k-c_{k-1})=c_j=\varphi(x).
\]
The constant term disappears when restricted to \(H_P\), because
\[
        (c_0\one_P)(h)=c_0\sum_{x\in P}h_x=0\qquad(h\in H_P).
\]
Therefore, in \(H_P^*\),
\[
        \varphi|_{H_P}
        =\sum_{k=1}^m(c_k-c_{k-1})\one_{U_k}|_{H_P},
\]
a nonnegative linear combination of restrictions of upper-set indicators.

Conversely, if \(U\subseteq P\) is upper, then \(\one_U|_{H_P}\) is
nonnegative on each generator of \(C_P\).  For \(x\prec y\),
\[
        \one_U(\delta_y-\delta_x)=\one_U(y)-\one_U(x)\ge0,
\]
because upperness forbids the case \(x\in U\) and \(y\notin U\).  Hence
\[
        C_P^*=
        \cone\{\one_U|_{H_P}:U\subseteq P\text{ upper}\}\subseteq H_P^*.
\]
Taking the dual cone again,
\[
\begin{aligned}
        C_P^{**}
        &=\{h\in H_P:\ell(h)\ge0\text{ for every }\ell\in C_P^*\} \\
        &=\{h\in H_P:\one_U(h)\ge0\text{ for every upper set }U\subseteq P\}.
\end{aligned}
\]
The second equality holds because \(C_P^*\) is generated by the upper-set
functionals.  Since \(\one_U(h)=h(U)\), the last set is precisely \(K_P\).  The
cone \(C_P\) is finitely generated, hence closed and polyhedral; by the
finite-dimensional bipolar theorem, \(C_P=C_P^{**}\).  Thus \(C_P=K_P\).

It remains to record the structural consequences used later.  The incidence
vectors of the cover edges linearly span \(H_P\).  To see this, fix a vertex
\(x_0\in P\).  Since the undirected Hasse graph is connected, for every
\(x\in P\) there is an undirected path
\[
        x_0=x_0',x_1',\ldots,x_n'=x.
\]
The signed sum of the incidence vectors along this path telescopes to
\[
        \delta_x-\delta_{x_0}.
\]
As \(x\) ranges over \(P\setminus\{x_0\}\), the vectors
\(\delta_x-\delta_{x_0}\) span
\[
        H_P=\left\{h\in\R^P:\sum_{x\in P}h_x=0\right\},
\]
because every \(h\in H_P\) can be written as
\[
        h=\sum_{x\ne x_0}h_x(\delta_x-\delta_{x_0}).
\]
Therefore \(K_P-K_P=H_P\).  Since \(K_P\) is a closed polyhedral cone whose
linear span is \(H_P\), it has nonempty relative interior in \(H_P\).

Finally, \(K_P\) is pointed.  Choose a function \(\rho:P\to\R\) which is
strictly increasing along covers; for instance, let \(\rho(x)\) be the maximum
length of a saturated chain ending at \(x\).  If \(x\prec y\), any longest
saturated chain ending at \(x\) extends to one ending at \(y\), so
\(\rho(y)\ge\rho(x)+1\).  Hence for every nonzero
\[
        h=\sum_{x\prec y}t_{xy}(\delta_y-\delta_x)\in K_P,
        \qquad t_{xy}\ge0,
\]
we have
\[
        \rho\cdot h=
        \sum_{x\prec y}t_{xy}\bigl(\rho(y)-\rho(x)\bigr)>0.
\]
Thus no nonzero element of \(K_P\) can also belong to \(-K_P\), and
\(K_P\cap(-K_P)=\{0\}\).
\end{proof}

\begin{lemma}[Cover edges are extreme]\label{lem:cover-extreme}
For every cover relation \(x\prec y\), the ray
\[
        \R_{\ge0}(\delta_y-\delta_x)
\]
is an extreme ray of \(K_P\).
\end{lemma}

\begin{proof}
By Lemma~\ref{lem:hasse-flow},
\[
        K_P=\cone\{\delta_v-\delta_u:u\prec v\}.
\]
It is enough to show that if
\[
        \delta_y-\delta_x
        =\sum_{u\prec v} t_{uv}(\delta_v-\delta_u),
        \qquad t_{uv}\ge0,
\]
then \(t_{uv}=0\) for every \((u,v)\ne(x,y)\).

Regard the numbers \(t_{uv}\) as a nonnegative flow on the directed Hasse
diagram, directed from \(u\) to \(v\) whenever \(u\prec v\).  The displayed
equality says that the resulting net vector has value \(-1\) at \(x\), value
\(+1\) at \(y\), and value \(0\) at every other vertex.  Thus \(x\) is the only
source, \(y\) is the only sink, and every other vertex satisfies flow
conservation.

This flow decomposes into directed paths from \(x\) to \(y\).  Starting at
\(x\), choose an outgoing edge with positive flow.  Whenever a positive-flow
path reaches a vertex \(z\ne y\), conservation at \(z\) implies that some
outgoing edge from \(z\) also has positive flow.  The Hasse diagram is acyclic,
so this process cannot continue forever and cannot form a directed cycle; hence
it must eventually reach \(y\).  Subtract the minimum flow along the obtained
path and repeat.  Since each step removes positive flow from at least one edge,
the procedure terminates after finitely many steps.

Since \(x\prec y\), there is no element strictly between \(x\) and \(y\).  A
directed path from \(x\) to \(y\) with more than one edge would contain such an
intermediate element.  Hence the only directed path from \(x\) to \(y\) is the
single edge \(x\prec y\).  Therefore no edge other than \(x\prec y\) can carry
positive flow, so all coefficients except possibly \(t_{xy}\) vanish.

Now suppose
\[
        \delta_y-\delta_x=a+b,
        \qquad a,b\in K_P.
\]
Writing \(a\) and \(b\) as nonnegative combinations of cover-edge vectors and
applying the preceding conclusion to their sum shows that both \(a\) and \(b\)
are nonnegative multiples of \(\delta_y-\delta_x\).  Hence
\(\R_{\ge0}(\delta_y-\delta_x)\) is an extreme ray of \(K_P\).
\end{proof}

\begin{proposition}[Hasse cycles give non-simplicial cones]\label{prop:cycle-nonsimplicial}
If the undirected Hasse graph of a connected finite poset \(P\) contains a
cycle, then \(K_P\) is not simplicial.
\end{proposition}

\begin{proof}
Let \(n=|P|\), and let
\[
        m=\#\{(x,y)\in P\times P:x\prec y\}
\]
be the number of cover edges of \(P\), equivalently the number of edges of the
undirected Hasse graph.  The space
\[
        H_P=\left\{h\in\R^P:\sum_{x\in P}h_x=0\right\}
\]
has dimension \(n-1\), since it is the kernel of the nonzero linear functional
\(h\mapsto\sum_{x\in P}h_x\) on the \(n\)-dimensional space \(\R^P\).

Because the undirected Hasse graph is connected and contains a cycle, its
number of edges satisfies \(m\ge n\).  Indeed, a connected graph with \(n\)
vertices is a tree exactly when it has \(n-1\) edges; a connected graph with a
cycle therefore has at least \(n\) edges.

For every cover relation \(x\prec y\), Lemma~\ref{lem:cover-extreme} shows that
\[
        \R_{\ge0}(\delta_y-\delta_x)
\]
is an extreme ray of \(K_P\).  These rays are pairwise distinct: the vector
\(\delta_y-\delta_x\) has a unique negative coordinate, at \(x\), and a unique
positive coordinate, at \(y\).  Thus two different cover edges cannot determine
the same ray.

Consequently, \(K_P\) has at least \(m\) distinct extreme rays, hence at least
\(n\) distinct extreme rays.  On the other hand, a simplicial cone in a
\(d\)-dimensional vector space has exactly \(d\) extreme rays.  Since
\(K_P\subseteq H_P\) and \(\dim H_P=n-1\), a simplicial cone in \(H_P\) would
have exactly \(n-1\) extreme rays.  This is impossible.  Therefore \(K_P\) is
not simplicial.
\end{proof}

\subsection{The local approximation supplied by RB}

\begin{lemma}[Strict order inside the simplex]\label{lem:strict-simplex}
Let \(A\subseteq\Delta_P\) be directed with supremum \(v\in\Delta_P\), and let \(u\in\Delta_P\).  If
\[
        v-u\in\relint(K_P),
\]
then there exists \(a\in A\) such that \(u\le a\).
\end{lemma}

\begin{proof}
By Lemma~\ref{lem:finite-directed-sup},
\[
        v(U)=\sup_{a\in A}a(U)
\]
for every upper set \(U\subseteq P\).  For \(U=\emptyset\) and \(U=P\),
all points of \(\Delta_P\) have the same value on \(U\), namely \(0\) and
\(1\), respectively.

Let \(U\) be a nonempty proper upper set.  Since the undirected Hasse graph is
connected, there is an undirected path from some point outside \(U\) to some
point inside \(U\).  Let \(x,y\) be the first adjacent pair on this path with
\(x\notin U\) and \(y\in U\).  Since \(U\) is upper, the cover direction must be
\(x\prec y\); otherwise \(y\prec x\) and \(y\in U\) would imply \(x\in U\).
Thus
\[
        (\delta_y-\delta_x)(U)=1.
\]
Hence the functional
\[
        \ell_U:H_P\to\mathbb R,\qquad \ell_U(h)=h(U),
\]
is nonzero.  It is also nonnegative on \(K_P\), by the definition of \(K_P\).

Since the Hasse graph is connected, the cover-edge incidence vectors span
\(H_P\).  By Lemma~\ref{lem:hasse-flow},
\[
        K_P=\operatorname{cone}\{\delta_y-\delta_x:x\prec y\},
\]
and therefore
\[
        K_P-K_P=H_P.
\]
Thus \(K_P\) is full-dimensional in \(H_P\).  We use the elementary fact that every nonzero linear functional which is nonnegative on a full-dimensional convex cone is strictly positive on its relative interior.  Indeed, if such a functional \(\ell\) vanished at \(w\in\relint(K_P)\), choose \(k\in K_P\) with \(\ell(k)>0\); then \(w-\varepsilon k\in K_P\) for sufficiently small \(\varepsilon>0\), but \(\ell(w-\varepsilon k)<0\), a contradiction.  Since \(v-u\in\relint(K_P)\), we obtain
\[
        0<\ell_U(v-u)=v(U)-u(U).
\]
Hence
\[
        u(U)<v(U)
\]
for every nonempty proper upper set \(U\).

There are only finitely many upper sets.  For each nonempty proper upper set
\(U\), choose \(a_U\in A\) such that
\[
        u(U)<a_U(U),
\]
which is possible because \(v(U)=\sup_{a\in A}a(U)\).  Since \(A\) is directed
and only finitely many \(a_U\)'s have been chosen, there exists \(a\in A\) such
that
\[
        a_U\le a
\]
for all such \(U\).  Therefore
\[
        u(U)<a_U(U)\le a(U)
\]
for every nonempty proper upper set \(U\).  For \(U=\emptyset\) and \(U=P\), we
also have \(u(U)=a(U)\).  Hence
\[
        u(U)\le a(U)
\]
for every upper set \(U\subseteq P\).  By the upper-set characterization of the
stochastic order, this means
\[
        u\le a.
\]
\end{proof}

\begin{proposition}[RB local approximation]\label{prop:RB-local-approximants}
Assume that \(\Vone(P)\) admits a directed finite-deflation approximation of the identity.  Let
\[
        L\Subset\relint(\Delta_P)
\]
be compact and convex with nonempty relative interior in the affine hull of \(\Delta_P\).  Then there are finite-valued \(K_P\)-monotone maps
\[
        Q_n:L\longrightarrow\Delta_P
\]
such that
\[
        \lim_{n\to\infty}\sup_{q\in L}\|Q_n(q)-q\|=0.
\]
\end{proposition}

\begin{proof}
Let \((r_i)_{i\in I}\) be a directed family of finite-range deflations with pointwise supremum \(\id\).  If \(H_P=\{0\}\), then \(\Delta_P\) is a single point and the conclusion is immediate.  Assume \(H_P\ne\{0\}\).  Choose \(e\in\relint(K_P)\).  Since \(e\in H_P\), the points \(q-te\) remain in the affine hull of \(\Delta_P\).  Since \(L\) is compact and contained in the relatively open set \(\relint(\Delta_P)\), there is \(\varepsilon_0>0\) such that
\[
        L-[0,\varepsilon_0]e\subseteq\relint(\Delta_P).
\]
In what follows take \(0<\varepsilon<\varepsilon_0\).

For \(p\in L\), set
\[
        u_p=p-\frac{3\varepsilon}{4}e,
        \qquad
        v_p=p-\frac{\varepsilon}{2}e.
\]
Then \(v_p-u_p=(\varepsilon/4)e\in\relint(K_P)\).  Since the family \((r_i)\) is directed in the pointwise order, the set \(\{r_i(v_p):i\in I\}\) is directed; its supremum is \(v_p\) because \(\sup_i r_i=\id\).  Applying Lemma \ref{lem:strict-simplex} gives \(i_p\in I\) such that
\[
        u_p\le r_{i_p}(v_p).
\]

Define
\[
        N_p=\{q\in L:
        u_p-(q-\varepsilon e)\in\relint(K_P),
        \quad
        q-v_p\in\relint(K_P)\}.
\]
The neighbourhood \(N_p\) is chosen so that \(q-\varepsilon e\) lies strictly below \(u_p\), while \(v_p\) lies strictly below \(q\).  The sets \(N_p\) are relatively open in \(L\) and cover \(L\).  Choose a finite subcover \(N_{p_1},\ldots,N_{p_N}\).  Directedness gives an index \(i\) above all \(i_{p_k}\).  If \(q\in N_{p_k}\), then
\[
        q-\varepsilon e
        \le u_{p_k}
        \le r_{i_{p_k}}(v_{p_k})
        \le r_i(q)
        \le q.
\]
Thus \(Q_\varepsilon=r_i|_L\) is finite-valued and \(K_P\)-monotone.  Moreover,
\[
        q-\varepsilon e\le Q_\varepsilon(q)\le q
        \qquad(q\in L).
\]
The order interval \([0,e]_{K_P}=K_P\cap(e-K_P)\) is compact.  Indeed, it is closed; if it were unbounded, a normalized unbounded sequence would have a nonzero limit lying in both \(K_P\) and \(-K_P\), contradicting pointedness.  From \(q-\varepsilon e\le Q_\varepsilon(q)\le q\) we have
\[
        q-Q_\varepsilon(q)\in \varepsilon[0,e]_{K_P}.
\]
Thus
\[
        \sup_{q\in L}\|Q_\varepsilon(q)-q\|
        \le \varepsilon\max_{z\in[0,e]_{K_P}}\|z\|.
\]
Taking \(\varepsilon_n\downarrow0\) gives the required sequence.
\end{proof}

\section{The Euclidean finite-step obstruction}\label{sec:euclidean}

In this section we prove the finite-dimensional obstruction used in the
cycle case.  We use throughout the Euclidean conventions of
Subsection~\ref{subsec:finite-dimensional-notation}.  Let
\(K\subseteq\R^d\) be a proper cone.  Thus \(K-K=\R^d\), the cone order
\(\le_K\) is a partial order, and standard finite-dimensional cone duality gives
\(\Int K^*\ne\emptyset\).  Define the matrix cone
\[
        \mathcal R_K=
        \cone\{\mathbf v\mathbf y^T:\mathbf v\in K,\ \mathbf y\in K^*\}
        \subseteq \operatorname{End}(\R^d).
\]

\subsection{The matrix cone of a proper cone}

The following finite-dimensional cone-theoretic fact is used in the obstruction
argument; we include the proof for completeness.

\begin{lemma}[Matrix-cone obstruction]\label{lem:matrix-cone-obstruction}
If \(K\) is a proper cone, then the cone \(\mathcal R_K\) is closed.  If
\[
        I_d\in\mathcal R_K,
\]
then \(K\) is simplicial.
\end{lemma}

\begin{proof}
Choose \(\mathbf a\in\Int K\) and \(\boldsymbol\alpha\in\Int K^*\); the latter
exists because a closed convex pointed cone with nonempty interior has a dual
cone with nonempty interior.  The bases
\[
        S=\{\mathbf x\in K:\boldsymbol\alpha\cdot\mathbf x=1\},
        \qquad
        S^*=\{\mathbf y\in K^*:\mathbf y\cdot\mathbf a=1\}
\]
are compact.  They are closed.  If \(S\) were unbounded, a normalized unbounded
sequence would have a nonzero limit \(\mathbf u\in K\) with
\(\boldsymbol\alpha\cdot\mathbf u=0\), contradicting
\(\boldsymbol\alpha\in\Int K^*\).  Similarly, if \(S^*\) were unbounded, a
normalized unbounded sequence would have a nonzero limit \(\mathbf y\in K^*\)
with \(\mathbf y\cdot\mathbf a=0\), contradicting \(\mathbf a\in\Int K\).
Every nonzero generator \(\mathbf v\mathbf y^T\) of \(\mathcal R_K\) is a
positive scalar multiple of \(\mathbf x\boldsymbol\zeta^T\) with
\(\mathbf x\in S\) and \(\boldsymbol\zeta\in S^*\).  The set
\[
        C_0=\conv\{\mathbf x\boldsymbol\zeta^T:
        \mathbf x\in S,
        \boldsymbol\zeta\in S^*\}
\]
is compact and lies in the affine hyperplane
\[
        \{A\in \operatorname{End}(\R^d):
        \boldsymbol\alpha\cdot(A\mathbf a)=1\}.
\]
The cone over \(C_0\) is closed.  Indeed, if \(t_nA_n\to A\) with
\(t_n\ge0\) and \(A_n\in C_0\), then
\[
        t_n=\boldsymbol\alpha\cdot(t_nA_n\mathbf a)
        \longrightarrow \boldsymbol\alpha\cdot(A\mathbf a),
\]
so \((t_n)\) is bounded and a convergent subsequence argument applies.  This
closed cone is exactly \(\mathcal R_K\).

Suppose now that \(I_d\in\mathcal R_K\).  Write
\[
        I_d=\sum_{j=1}^m \mathbf v_j\mathbf y_j^T,
        \qquad
        \mathbf v_j\in K,
        \quad
        \mathbf y_j\in K^*.
\]
Discarding zero terms, set
\[
        \mathbf p_j=\frac{\mathbf v_j}{\boldsymbol\alpha\cdot\mathbf v_j},
        \qquad
        \boldsymbol\theta_j=(\boldsymbol\alpha\cdot\mathbf v_j)\mathbf y_j.
\]
For every \(\mathbf x\in S\),
\[
        \mathbf x=\sum_j(\boldsymbol\theta_j\cdot\mathbf x)\mathbf p_j,
        \qquad
        \boldsymbol\theta_j\cdot\mathbf x\ge0,
        \qquad
        \sum_j\boldsymbol\theta_j\cdot\mathbf x
        =\boldsymbol\alpha\cdot\mathbf x=1.
\]
Therefore \(S=\conv\{\mathbf p_1,\ldots,\mathbf p_m\}\).  Let \(E\) be the set
of extreme points of \(S\).  Then \(E\) is finite, and every point of \(E\) is
one of the \(\mathbf p_j\)'s.  If \(\mathbf q\in E\) and
\(\boldsymbol\theta_j\cdot\mathbf q>0\), the above convex representation of
\(\mathbf q\) forces \(\mathbf p_j=\mathbf q\).  For \(\mathbf q\in E\), define
\[
        \boldsymbol\beta_{\mathbf q}
        =\sum_{\{j:\mathbf p_j=\mathbf q\}}\boldsymbol\theta_j.
\]
Then for \(\mathbf q,\mathbf r\in E\),
\[
        \boldsymbol\beta_{\mathbf q}\cdot\mathbf r
        =
        \begin{cases}
        1,&\mathbf q=\mathbf r,\\
        0,&\mathbf q\ne\mathbf r.
        \end{cases}
\]
Thus the extreme points in \(E\) are linearly independent.  Since
\(K=\cone(S)=\cone(E)\) and \(K-K=\R^d\), the set \(E\) spans \(\R^d\).  Hence
\(E\) is a basis of \(\R^d\).  The cone \(K\) is therefore generated by the
linearly independent rays \(\R_{\ge0}\mathbf q\), \(\mathbf q\in E\).  By the
definition above, \(K\) is simplicial.
\end{proof}

\begin{lemma}[Cone-valued integration]\label{lem:cone-valued-integral}
Let \(C\) be a closed convex cone in a finite-dimensional real vector space
\(E\), let \(U\subseteq\R^n\) be measurable, and let
\(F:U\to E\) be integrable.  If \(F(\mathbf x)\in C\) for almost every \(\mathbf x\in U\), then
\[
        \int_U F(\mathbf x)\dd x\in C.
\]
\end{lemma}
\begin{proof}
Put
\[
        a=\int_U F(\mathbf x)\dd x.
\]
Suppose that \(a\notin C\).  Since \(C\) is closed and convex, the separation
theorem gives a linear functional \(\ell\in E^*\) such that
\[
        \ell(a)<0
        \qquad\text{and}\qquad
        \ell(b)\ge0\quad(b\in C).
\]
Indeed, this follows from strong separation and the fact that \(C\) is a cone.
Since \(F(\mathbf x)\in C\) almost everywhere, we have
\[
        \ell(F(\mathbf x))\ge0
\]
almost everywhere.  By linearity of finite-dimensional vector-valued integration,
\[
        \ell(a)
        =
        \ell\left(\int_U F(\mathbf x)\dd x\right)
        =
        \int_U \ell(F(\mathbf x))\dd x
        \ge0,
\]
contradicting \(\ell(a)<0\).  Hence \(a\in C\).
\end{proof}

\subsection{Lipschitz epigraphs for upper sets}
For a proper cone \(K\subseteq\R^d\), fix once and for all
\(\mathbf e\in\Int K\) and \(\boldsymbol\eta\in\Int K^*\) such that
\(\boldsymbol\eta\cdot \mathbf e=1\).  Such a pair exists: choose
\(\mathbf e_0\in\Int K\).  Since \(K\) is proper, \(\Int K^*\ne\emptyset\).  If
\(\boldsymbol\eta_0\in\Int K^*\), then
\(\boldsymbol\eta_0\cdot \mathbf e_0>0\); set
\(\mathbf e=\mathbf e_0\) and
\(\boldsymbol\eta=\boldsymbol\eta_0/(\boldsymbol\eta_0\cdot \mathbf e_0)\).  Put
\[
        W=\{\mathbf z\in\R^d:\boldsymbol\eta\cdot \mathbf z=0\},
        \qquad
        \R^d=\R\mathbf e\oplus W.
\]
Thus every \(\mathbf x\in\R^d\) is uniquely written as
\[
        \mathbf x=t\mathbf e+\mathbf z,
        \qquad
        t=\boldsymbol\eta\cdot \mathbf x,
        \quad \mathbf z\in W.
\]
We choose the Lebesgue measure \(\dd z\) on \(W\) normalized so that, under the linear coordinates
\[
        \R\times W\longrightarrow\R^d,
        \qquad (t,\mathbf z)\longmapsto t\mathbf e+\mathbf z,
\]
the ambient Lebesgue measure satisfies
\[
        \dd x=\dd t\,\dd z.
\]
Equivalently, this absorbs the constant Jacobian of the above linear isomorphism into the normalization of \(\dd z\).  This rescaling of \(\dd z\) does not affect the definition of \(\nabla_W\), which is always computed using the Euclidean inner product inherited from \(\R^d\).

For later use, we define the modified transverse gradient.  Let \(V\subseteq W\) be open, and let \(f:V\to\R\) be differentiable at \(\mathbf z\in V\).  We identify \(\boldsymbol\eta\in K^*\) with its Euclidean representing vector.  Define
\[
        \nabla_{\mathbf e,W}f(\mathbf z)
        =\nabla_W f(\mathbf z)-\bigl(\nabla_W f(\mathbf z)\cdot \mathbf e\bigr)\boldsymbol\eta .
\]
Equivalently, if \(\mathbf h=\tau\mathbf e+\mathbf w\) with \(\mathbf w\in W\), then
\[
        \nabla_{\mathbf e,W}f(\mathbf z)\cdot \mathbf h
        =\nabla_W f(\mathbf z)\cdot \mathbf w .
\]
The vector \(\nabla_{\mathbf e,W}f(\mathbf z)\) is the ambient representative of the \(W\)-gradient of \(f\) adapted to the non-orthogonal splitting \(\mathbb R^d=\R\mathbf e\oplus W\); it agrees with \(\nabla_W f(\mathbf z)\) on \(W\)-directions and vanishes on the \(\mathbf e\)-direction.  Equivalently, for every \(\mathbf w\in W\),
\[
        \bigl(\boldsymbol\eta-\nabla_{\mathbf e,W}f(\mathbf z)\bigr)\cdot
        \bigl(Df_{\mathbf z}(\mathbf w)\mathbf e+\mathbf w\bigr)=0.
\]
Thus \(\boldsymbol\eta-\nabla_{\mathbf e,W}f(\mathbf z)\) is orthogonal to the tangent space of the graph \(\mathbf z\mapsto f(\mathbf z)\mathbf e+\mathbf z\).
When \(f\) is Lipschitz, this definition is used at points of differentiability, which exist almost everywhere by Rademacher's theorem.

Let \(a<b\), let \(U\subseteq W\) be open, and put
\[
        C=(a,b)\mathbf e+U.
\]
A set \(E\subseteq C\) is called \(K\)-upper in \(C\) if it is an upper set for the cone order \(\le_K\), relative to \(C\); equivalently,
\[
        \mathbf x\in E,
        \quad \mathbf y\in C,
        \quad \mathbf y-\mathbf x\in K
        \quad\Longrightarrow\quad
        \mathbf y\in E.
\]

\begin{lemma}[Upper sets are Lipschitz epigraphs]\label{lem:upper-lipschitz-epigraph}
Let \(U\subseteq W\) be open, let \(a<b\), and put
\[
        C=(a,b)\mathbf e+U.
\]
Let \(E\subseteq C\) be \(K\)-upper in \(C\).  Then \(E\) is Lebesgue measurable,
and there is a Lipschitz function \(f_E:U\to[a,b]\) such that, with
\[
        G_{f_E}=\{t\mathbf e+\mathbf z:\mathbf z\in U,\ f_E(\mathbf z)<t<b\},
\]
the symmetric difference \(E\triangle G_{f_E}\) has Lebesgue measure zero in
\(C\).
\end{lemma}

\begin{proof}
For \(\mathbf z\in U\), define the vertical section
\[
        E_{\mathbf z}=\{t\in(a,b):t\mathbf e+\mathbf z\in E\}.
\]
Since \(\mathbf e\in K\) and \(E\) is \(K\)-upper in \(C\), each \(E_{\mathbf z}\) is an upper
subset of \((a,b)\).  Define
\[
        f_E(\mathbf z)=\inf E_{\mathbf z},
\]
with the convention that \(f_E(\mathbf z)=b\) if \(E_{\mathbf z}=\emptyset\).  Then
\(f_E(\mathbf z)\in[a,b]\), and \(E_{\mathbf z}\) differs from \((f_E(\mathbf z),b)\) by at most the
single point \(f_E(\mathbf z)\).  In particular, if
\[
        f_E(\mathbf z)<t<b,
\]
then \(t\in E_{\mathbf z}\), equivalently \(t\mathbf e+\mathbf z\in E\).

Since \(\mathbf e\in\Int K\), there is \(\rho>0\) such that
\[
        \mathbf e+B(0,\rho)\subseteq K.
\]
Choose \(L>1/\rho\).  Then, for every \(\mathbf w\in W\),
\[
        L\|\mathbf w\|\mathbf e+\mathbf w\in K.
\]
Indeed, if \(\mathbf w\ne0\), then
\[
        \left\|\frac{\mathbf w}{L\|\mathbf w\|}\right\|<\rho,
\]
so
\[
        \mathbf e+\frac{\mathbf w}{L\|\mathbf w\|}\in K,
\]
and multiplying by \(L\|\mathbf w\|>0\) gives the claim.  The case \(\mathbf w=0\) is immediate.

Let \(\mathbf z,\mathbf z'\in U\), and put \(\mathbf w=\mathbf z'-\mathbf z\).  Suppose first that
\[
        f_E(\mathbf z)+L\|\mathbf w\|<b.
\]
Choose \(t\) with
\[
        f_E(\mathbf z)<t<b-L\|\mathbf w\|.
\]
Then \(t\mathbf e+\mathbf z\in E\), and
\[
        (t+L\|\mathbf w\|)\mathbf e+\mathbf z'-(t\mathbf e+\mathbf z)
        =
        L\|\mathbf w\|\mathbf e+\mathbf w\in K.
\]
Moreover \((t+L\|\mathbf w\|)\mathbf e+\mathbf z'\in C\).  Since \(E\) is \(K\)-upper in \(C\),
\[
        (t+L\|\mathbf w\|)\mathbf e+\mathbf z'\in E.
\]
Hence
\[
        f_E(\mathbf z')\le t+L\|\mathbf w\|.
\]
Letting \(t\downarrow f_E(\mathbf z)\), we obtain
\[
        f_E(\mathbf z')\le f_E(\mathbf z)+L\|\mathbf z'-\mathbf z\|.
\]
If instead \(f_E(\mathbf z)+L\|\mathbf z'-\mathbf z\|\ge b\), the same inequality follows from the
trivial bound \(f_E(\mathbf z')\le b\).  Thus
\[
        f_E(\mathbf z')\le f_E(\mathbf z)+L\|\mathbf z'-\mathbf z\|.
\]
Interchanging \(\mathbf z\) and \(\mathbf z'\) gives
\[
        |f_E(\mathbf z')-f_E(\mathbf z)|\le L\|\mathbf z'-\mathbf z\|.
\]
Therefore \(f_E\) is Lipschitz.

For every \(\mathbf z\in U\),
\[
        E_{\mathbf z}\triangle(f_E(\mathbf z),b)\subseteq \{f_E(\mathbf z)\}.
\]
Consequently
\[
        E\triangle G_{f_E}
        \subseteq
        \{f_E(\mathbf z)\mathbf e+\mathbf z:\mathbf z\in U\}\cap C.
\]
The set on the right is the image, under the linear isomorphism
\((t,\mathbf z)\mapsto t\mathbf e+\mathbf z\), of the usual graph of the Lipschitz function \(f_E\).
Hence it has \(d\)-dimensional Lebesgue measure zero.  Since Lebesgue measure is
complete, every subset of this null graph is measurable and null.  Thus
\(E\triangle G_{f_E}\) is measurable and null.

Finally, \(f_E\) is continuous and \(U\) is open, so \(G_{f_E}\) is Borel.
Since
\[
        E=G_{f_E}\triangle(E\triangle G_{f_E}),
\]
the set \(E\) is Lebesgue measurable and differs from \(G_{f_E}\) by a null set.
\end{proof}

\begin{lemma}[Epigraph integration formula]\label{lem:epigraph-integration}
Let \(U\subseteq W\) be open, let \(a<b\), and put
\[
        C=(a,b)\mathbf e+U.
\]
Let \(f:U\to[a,b]\) be Lipschitz and put
\[
        G_f=\{t\mathbf e+\mathbf z:\mathbf z\in U,\ f(\mathbf z)<t<b\}.
\]
For \(\psi\in C_c^\infty(C)\), extended by zero outside \(C\), one has
\[
        -\int_{G_f} \nabla\psi(\mathbf x)\dd x
        =
        \int_U \psi(f(\mathbf z)\mathbf e+\mathbf z)\bigl(\boldsymbol\eta-\nabla_{\mathbf e,W}f(\mathbf z)\bigr)\dd z.
\]
On the null set where \(f\) is not differentiable, \(\nabla_{\mathbf e,W}f(\mathbf z)\) may be
chosen arbitrarily.  Moreover, if \(f=f_E\) is obtained from a \(K\)-upper set
\(E\subseteq C\) as in Lemma~\ref{lem:upper-lipschitz-epigraph}, then
\[
        \boldsymbol\eta-\nabla_{\mathbf e,W}f(\mathbf z)\in K^*
\]
for almost every \(\mathbf z\in U\) with \(f(\mathbf z)\in(a,b)\).
\end{lemma}

\begin{proof}
By Rademacher's theorem \cite[Section~3.1.2]{EvansGariepy2015}, \(f\) is
differentiable almost everywhere on \(W\).  We choose \(\nabla_{\mathbf e,W}f(\mathbf z)\)
arbitrarily on the null set where \(f\) is not differentiable.

Let \(\mathbf h\in\mathbb R^d\), and write it uniquely as
\[
        \mathbf h=\tau\mathbf e+\mathbf w,
        \qquad \tau\in\mathbb R,\quad \mathbf w\in W.
\]
Then \(\tau=\boldsymbol\eta\cdot \mathbf h\).  For fixed \(t\), write
\[
        \psi_t(\mathbf z)=\psi(t\mathbf e+\mathbf z),
\]
and let \(\nabla_W\psi_t(\mathbf z)\) denote its \(W\)-gradient.  Since
\[
        \partial_t\psi(t\mathbf e+\mathbf z)=\nabla\psi(t\mathbf e+\mathbf z)\cdot \mathbf e,
\]
we have
\[
        \nabla\psi(t\mathbf e+\mathbf z)\cdot \mathbf h
        =
        \tau\,\partial_t\psi(t\mathbf e+\mathbf z)
        +
        \nabla_W\psi_t(\mathbf z)\cdot \mathbf w.
\]
Using the coordinates \(\mathbf x=t\mathbf e+\mathbf z\), the normalization \(\dd x=\dd t\,\dd z\), and
Fubini's theorem, we get
\[
\begin{aligned}
 -\int_{G_f}\nabla\psi(\mathbf x)\cdot \mathbf h\dd x
 &=
 -\int_U\int_{f(\mathbf z)}^b
        \left(
        \tau\,\partial_t\psi(t\mathbf e+\mathbf z)
        +
        \nabla_W\psi_t(\mathbf z)\cdot \mathbf w
        \right)
        \dd t\,\dd z.
\end{aligned}
\]

The \(t\)-derivative term gives
\[
\begin{aligned}
-\tau\int_U\int_{f(\mathbf z)}^b\partial_t\psi(t\mathbf e+\mathbf z)\dd t\,\dd z
&=
-\tau\int_U
        \bigl(\psi(b\mathbf e+\mathbf z)-\psi(f(\mathbf z)\mathbf e+\mathbf z)\bigr)
        \dd z  \\
&=
\tau\int_U\psi(f(\mathbf z)\mathbf e+\mathbf z)\dd z,
\end{aligned}
\]
because \(\psi\) has zero trace at the fixed boundary level \(t=b\).

For the \(W\)-term, set
\[
        H(\mathbf z)=\int_{f(\mathbf z)}^b\psi(t\mathbf e+\mathbf z)\dd t.
\]
Then \(H\) is Lipschitz on \(U\).  Moreover, since \(\psi\) has compact support
in \(C\), the support of \(H\) is contained in a compact subset of \(U\).
Thus the zero extension of \(H\) to \(W\) is a compactly supported Lipschitz
function.  At almost every point where \(f\) is differentiable, the Leibniz
rule gives
\[
        \nabla_W H(\mathbf z)\cdot \mathbf w
        =
        \int_{f(\mathbf z)}^b \nabla_W\psi_t(\mathbf z)\cdot \mathbf w\dd t
        -
        \psi(f(\mathbf z)\mathbf e+\mathbf z)\nabla_W f(\mathbf z)\cdot \mathbf w.
\]
Since the zero extension of \(H\) is compactly supported and Lipschitz on \(W\),
\[
        \int_W \nabla_W H(\mathbf z)\cdot \mathbf w\dd z=0.
\]
Therefore
\[
        \int_U\int_{f(\mathbf z)}^b \nabla_W\psi_t(\mathbf z)\cdot \mathbf w\dd t\,\dd z
        =
        \int_U\psi(f(\mathbf z)\mathbf e+\mathbf z)\nabla_W f(\mathbf z)\cdot \mathbf w\dd z.
\]

Combining the \(t\)-term and the \(W\)-term, we obtain
\[
        -\int_{G_f}\nabla\psi(\mathbf x)\cdot \mathbf h\dd x
        =
        \int_U\psi(f(\mathbf z)\mathbf e+\mathbf z)
        \left(\tau-\nabla_W f(\mathbf z)\cdot \mathbf w\right)
        \dd z.
\]
By the defining property of \(\nabla_{\mathbf e,W}f(\mathbf z)\),
\[
        \bigl(\boldsymbol\eta-\nabla_{\mathbf e,W}f(\mathbf z)\bigr)\cdot \mathbf h
        =
        \tau-\nabla_W f(\mathbf z)\cdot \mathbf w.
\]
Hence
\[
        -\int_{G_f}\nabla\psi(\mathbf x)\cdot \mathbf h\dd x
        =
        \int_U\psi(f(\mathbf z)\mathbf e+\mathbf z)
        \bigl(\boldsymbol\eta-\nabla_{\mathbf e,W}f(\mathbf z)\bigr)\cdot \mathbf h\dd z.
\]
Since this holds for every \(\mathbf h\in\mathbb R^d\), the vector identity follows.

Now assume that \(f=f_E\) comes from a \(K\)-upper set \(E\subseteq C\).  Let
\(\mathbf z\in U\) be a point where \(f\) is differentiable and \(f(\mathbf z)\in(a,b)\).
Take \(\mathbf h=\tau\mathbf e+\mathbf w\in K\), with \(\mathbf w\in W\).  Since \(\boldsymbol\eta\in K^*\),
\[
        \tau=\boldsymbol\eta\cdot \mathbf h\ge0.
\]
For \(s>0\) sufficiently small, we have
\[
        \mathbf z+s\mathbf w\in U
        \qquad\text{and}\qquad
        f(\mathbf z)+s\tau<b.
\]
For every \(\delta>0\) with \(f(\mathbf z)+\delta+s\tau<b\), we have
\[
        (f(\mathbf z)+\delta)\mathbf e+\mathbf z\in E.
\]
Since \(s\mathbf h\in K\), upperness gives
\[
        (f(\mathbf z)+\delta+s\tau)\mathbf e+(\mathbf z+s\mathbf w)\in E.
\]
Therefore
\[
        f(\mathbf z+s\mathbf w)\le f(\mathbf z)+\delta+s\tau.
\]
Letting \(\delta\downarrow0\), we obtain
\[
        f(\mathbf z+s\mathbf w)-f(\mathbf z)\le s\tau.
\]
Dividing by \(s>0\) and then letting \(s\downarrow0\), we get
\[
        \nabla_W f(\mathbf z)\cdot \mathbf w\le \tau.
\]
Thus
\[
        \bigl(\boldsymbol\eta-\nabla_{\mathbf e,W}f(\mathbf z)\bigr)\cdot \mathbf h
        =
        \tau-\nabla_W f(\mathbf z)\cdot \mathbf w
        \ge0.
\]
Since \(\mathbf h\in K\) was arbitrary,
\[
        \boldsymbol\eta-\nabla_{\mathbf e,W}f(\mathbf z)\in K^*.
\]
This holds at every differentiability point \(\mathbf z\) with \(f(\mathbf z)\in(a,b)\), hence
almost everywhere on that set.
\end{proof}

\subsection{Finite-valued monotone maps}

Given an open set \(O\subseteq\R^d\), a measurable map \(Q:O\to\R^d\), and \(\psi\in C_c^\infty(O)\), define the matrix-valued flux
\[
        T_Q(\psi)=-\int_O Q(\mathbf x)\nabla\psi(\mathbf x)^T\dd x,
\]
whenever the integral is defined.  Here \(Q(\mathbf x)\nabla\psi(\mathbf x)^T\) is the rank-one matrix \(\mathbf h\mapsto Q(\mathbf x)(\nabla\psi(\mathbf x)\cdot \mathbf h)\).  In the lemma below \(O\) is the cylinder \(C\).

\begin{lemma}[Finite-valued monotone maps give matrix-cone fluxes]\label{lem:finite-step-flux}
Let \(\Omega\subseteq\R^d\) be open, let \(Q:\Omega\to\R^d\) be finite-valued
and \(K\)-monotone, and let
\[
        C=(a,b)\mathbf e+U
\]
be a cylinder with \(\overline C\subseteq\Omega\), where \(U\subseteq W\) is
bounded and open.  If \(0\le\psi\in C_c^\infty(C)\), then \(Q\) is measurable
on \(C\) and
\[
        T_Q(\psi)\in\mathcal R_K.
\]
\end{lemma}

\begin{proof}
Let \(\mathbf q_1,\ldots,\mathbf q_N\) be the distinct values of \(Q\) on \(\Omega\).  If
\(N=1\), then \(T_Q(\psi)=0\in\mathcal R_K\).  Assume \(N\ge2\).  Choose
\(\boldsymbol\lambda\in\Int K^*\) such that the numbers \(\boldsymbol\lambda\cdot \mathbf q_i\) are pairwise
distinct.  For \(\mathbf q_i\ne \mathbf q_j\), the condition \(\boldsymbol\lambda\cdot \mathbf q_i=\boldsymbol\lambda\cdot \mathbf q_j\)
defines a proper hyperplane in \(\R^d\).  Since \(\Int K^*\) is open and nonempty, it is not covered by finitely many such hyperplanes.  Relabel the values so that
\[
        \boldsymbol\lambda\cdot \mathbf q_1<\boldsymbol\lambda\cdot \mathbf q_2<\cdots<\boldsymbol\lambda\cdot \mathbf q_N.
\]
Choose numbers \(s_i\) with
\[
        \boldsymbol\lambda\cdot \mathbf q_i<s_i<\boldsymbol\lambda\cdot \mathbf q_{i+1},
        \qquad 1\le i<N,
\]
and define
\[
        E_i=\{\mathbf x\in\Omega:\boldsymbol\lambda\cdot Q(\mathbf x)>s_i\}.
\]
Each \(E_i\) is \(K\)-upper in \(\Omega\), because \(Q\) is \(K\)-monotone and
\(\boldsymbol\lambda\in K^*\).  Hence \(E_i\cap C\) is \(K\)-upper in \(C\).  By
Lemma~\ref{lem:upper-lipschitz-epigraph}, take the threshold functions
\(f_i=f_{E_i\cap C}:U\to[a,b]\).  Then \(E_i\cap C\) differs from
\[
        G_i=\{t\mathbf e+\mathbf z:\mathbf z\in U,\ f_i(\mathbf z)<t<b\}
\]
by a null set, and for every fixed \(\mathbf z\in U\) and every \(t\in(a,b)\) with \(t\ne f_i(\mathbf z)\), membership is exact:
\[
        t\mathbf e+\mathbf z\in E_i\quad\Longleftrightarrow\quad t>f_i(\mathbf z).
\]
Since
\[
        E_1\supseteq E_2\supseteq\cdots\supseteq E_{N-1},
\]
the threshold functions satisfy
\[
        f_1\le f_2\le\cdots\le f_{N-1}.
\]
In particular, the sets \(E_i\cap C\) are measurable, the phase sets of \(Q\) in \(C\) are measurable, and the layer formula a.e. on \(C\) is
\[
        Q=\mathbf q_1+\sum_{i=1}^{N-1}(\mathbf q_{i+1}-\mathbf q_i)\one_{E_i}.
\]
The constant term gives no contribution because \(\int_C \nabla\psi(\mathbf x)\dd x=0\).  By Lemma~\ref{lem:epigraph-integration},
\[
        T_Q(\psi)
        =\int_U
        \sum_{i=1}^{N-1}
        \psi(f_i(\mathbf z)\mathbf e+\mathbf z)(\mathbf q_{i+1}-\mathbf q_i)
        \bigl(\boldsymbol\eta-\nabla_{\mathbf e,W}f_i(\mathbf z)\bigr)^T
        \dd z.
\]
It remains to show that the integrand belongs to \(\mathcal R_K\) for almost
every \(\mathbf z\).

Fix a point \(\mathbf z\) at which all \(f_i\)'s are differentiable.  Partition
\(\{1,\ldots,N-1\}\) into maximal consecutive blocks on which the values
\(f_i(\mathbf z)\) are equal.  Let \(r,\ldots,s\) be such a block and write
\[
        f_r(\mathbf z)=\cdots=f_s(\mathbf z)=t.
\]
If \(\psi(t\mathbf e+\mathbf z)=0\), the block contributes zero.  Otherwise \(t\in(a,b)\).  For
\(r\le i<s\), the function \(f_{i+1}-f_i\) is nonnegative and has value \(0\)
at \(\mathbf z\); since both functions are differentiable at \(\mathbf z\), this forces
\[
        \nabla_{\mathbf e,W}f_r(\mathbf z)=\nabla_{\mathbf e,W}f_{r+1}(\mathbf z)=\cdots=\nabla_{\mathbf e,W}f_s(\mathbf z).
\]
Thus the block contribution is
\[
        \psi(t\mathbf e+\mathbf z)(\mathbf q_{s+1}-\mathbf q_r)
        \bigl(\boldsymbol\eta-\nabla_{\mathbf e,W}f_r(\mathbf z)\bigr)^T.
\]
By Lemma~\ref{lem:epigraph-integration}, \(\boldsymbol\eta-\nabla_{\mathbf e,W}f_r(\mathbf z)\in K^*\).  We claim
that \(\mathbf q_{s+1}-\mathbf q_r\in K\).  By maximality of the block, choose numbers
\(t_-<t<t_+\) so close to \(t\) that
\[
        t_- > f_{r-1}(\mathbf z)\quad(r>1),
        \qquad
        t_+ < f_{s+1}(\mathbf z)\quad(s<N-1),
\]
with the evident omissions when \(r=1\) or \(s=N-1\).  The choices of \(t_-\) and \(t_+\) avoid all threshold levels \(f_i(\mathbf z)\).  By the pointwise section statement in Lemma~\ref{lem:upper-lipschitz-epigraph}, the membership pattern in the sets \(E_i\) is exact at these two points, not merely almost everywhere, and
\[
        Q(t_-\mathbf e+\mathbf z)=\mathbf q_r,
        \qquad
        Q(t_+\mathbf e+\mathbf z)=\mathbf q_{s+1}.
\]
Since \((t_+-t_-)\mathbf e\in K\) and \(Q\) is \(K\)-monotone, it follows that
\[
        \mathbf q_{s+1}-\mathbf q_r\in K.
\]
Therefore each nonzero block contribution is a nonnegative scalar multiple of
a generator of \(\mathcal R_K\).  The whole integrand lies in \(\mathcal R_K\)
for almost every \(\mathbf z\), and Lemma~\ref{lem:cone-valued-integral} gives
\(T_Q(\psi)\in\mathcal R_K\).
\end{proof}

\begin{proposition}[Euclidean finite-step obstruction]\label{prop:euclidean-obstruction}
Let \(K\subseteq\R^d\) be a closed convex pointed cone with nonempty interior.
If \(K\) is not simplicial, then, for every nonempty open set
\(\Omega\subseteq\R^d\), there is no sequence of finite-valued \(K\)-monotone
maps
\[
        Q_n:\Omega\to\R^d
\]
which converges to the identity uniformly on compact subsets of \(\Omega\).
\end{proposition}

\begin{proof}
Suppose, for a contradiction, that such a sequence \((Q_n)\) exists.

By Lemma~\ref{lem:matrix-cone-obstruction},
\[
        I_d\notin\mathcal R_K.
\]
Since \(\Omega\) is nonempty and open, we may choose a cylinder
\[
        C=(a,b)\mathbf e+U\Subset\Omega,
\]
where \(U\subseteq W\) is bounded and open.  Choose
\[
        0\le\psi\in C_c^\infty(C)
\]
with
\[
        \int_C\psi(\mathbf x)\dd x=1.
\]
By Lemma~\ref{lem:finite-step-flux},
\[
        -\int_C Q_n(\mathbf x)\nabla\psi(\mathbf x)^T\dd x\in\mathcal R_K
\]
for every \(n\).  Since \(Q_n\to\id\) uniformly on \(\supp\psi\),
\[
\begin{aligned}
\left\|
        \int_C (Q_n(\mathbf x)-\mathbf x)\nabla\psi(\mathbf x)^T\dd x
\right\|
&\le
        \sup_{\mathbf x\in\supp\psi}\|Q_n(\mathbf x)-\mathbf x\|
        \int_C\|\nabla\psi(\mathbf x)\|\dd x
        \longrightarrow0 .
\end{aligned}
\]
Hence
\[
        -\int_C Q_n(\mathbf x)\nabla\psi(\mathbf x)^T\dd x
        \longrightarrow
        -\int_C \mathbf x\nabla\psi(\mathbf x)^T\dd x.
\]
Since \(\mathcal R_K\) is closed,
\[
        -\int_C \mathbf x\nabla\psi(\mathbf x)^T\dd x\in\mathcal R_K.
\]
For every \(\mathbf h\in\R^d\), coordinatewise integration by parts gives
\[
\begin{aligned}
        \left(-\int_C \mathbf x\nabla\psi(\mathbf x)^T\dd x\right)\mathbf h
        &=-\int_C \mathbf x\,\bigl(\nabla\psi(\mathbf x)\cdot \mathbf h\bigr)\dd x \\
        &=\mathbf h\int_C\psi(\mathbf x)\dd x
        =\mathbf h .
\end{aligned}
\]
There is no boundary term because \(\psi\) has compact support in \(C\).  Thus
\[
        -\int_C \mathbf x\nabla\psi(\mathbf x)^T\dd x=I_d,
\]
contradicting \(I_d\notin\mathcal R_K\).
\end{proof}

\section{Proof of the classification}\label{sec:classification}

\subsection{The Hasse-cycle obstruction}

\begin{theorem}[Hasse-cycle obstruction]\label{thm:cycle-obstruction}
Let \(P\) be a finite poset with a least element.  If the undirected Hasse graph of \(P\) contains a cycle, then \(\Vone(P)\) does not admit a directed finite-deflation approximation of the identity.  In particular, \(\Vone(P)\) is not an RB-domain.
\end{theorem}

\begin{proof}
By Corollary \ref{cor:least-connected}, the Hasse graph of \(P\) is connected.  Choose \(L\Subset\relint(\Delta_P)\) compact, convex, and with nonempty interior in the affine hull of \(\Delta_P\); for instance, take a sufficiently small closed ball in that affine hyperplane.  If \(\Vone(P)\) admitted a directed finite-deflation approximation of the identity, Proposition \ref{prop:RB-local-approximants} would give finite-valued \(K_P\)-monotone maps \(Q_n:L\to\Delta_P\) converging uniformly to the identity on \(L\).

Choose \(p_*\in\relint(L)\).  The set \(\Omega=\relint(L)-p_*\subseteq H_P\) is a nonempty open convex subset of \(H_P\).  In orthonormal coordinates on \(H_P\), the translated maps
\[
        \widehat Q_n(\mathbf x)=Q_n(p_*+\mathbf x)-p_*
\]
are finite-valued \(K_P\)-monotone maps and converge to the identity uniformly on compact subsets of \(\Omega\).  By Proposition \ref{prop:cycle-nonsimplicial}, the cone \(K_P\) is non-simplicial.  This contradicts Proposition \ref{prop:euclidean-obstruction}.
\end{proof}

\subsection{The positive tree case and the classification}

We use Goubault-Larrecq's  finite-tree result: for every finite tree domain \(T\), the domain \(\Vone(T)\) is a countably-based bc-domain \cite[Lemma~6.8]{GoubaultLarrecq2012}.  Here a bc-domain is a bounded-complete continuous dcpo; we use the standard fact that countably-based bc-domains are RB-domains \cite{Gierz2003}.

\begin{theorem}[Classification for finite posets]\label{thm:finite-poset-classification}
Let \(P\) be a finite nonempty poset.  Then
\[
        \Vone(P)\text{ is an RB-domain}
        \quad\Longleftrightarrow\quad
        P\text{ has a least element and the undirected Hasse graph of }P\text{ is a tree}.
\]
Equivalently, by Lemma~\ref{lem:finite-tree-equivalence}, \(\Vone(P)\) is an RB-domain exactly for finite tree domains.
\end{theorem}

\begin{proof}
First suppose that \(P\) has a least element and that its undirected Hasse graph is a tree.  By Lemma~\ref{lem:finite-tree-equivalence}, \(P\) is a finite tree domain.  Goubault-Larrecq proves that, for every finite tree domain \(T\), \(\Vone(T)\) is a countably-based bc-domain \cite[Lemma~6.8]{GoubaultLarrecq2012}.  Since countably-based bc-domains are RB-domains, \(\Vone(P)\) is an RB-domain.

Conversely, suppose \(\Vone(P)\) is an RB-domain.  Proposition \ref{prop:least-element-obstruction} implies that \(P\) has a least element.  If the Hasse graph had a cycle, Theorem~\ref{thm:cycle-obstruction} would give a contradiction.  Hence the Hasse graph is connected and acyclic, so it is a tree.  The equivalence with finite tree domains follows from Lemma~\ref{lem:finite-tree-equivalence}.
\end{proof}

\begin{corollary}[RB is not preserved by \(\Vone\)]\label{cor:RB-not-preserved}
The  probabilistic powerdomain \(\Vone\) does not preserve RB-domains.
\end{corollary}

\begin{proof}
Let
\[
        \Dia=\{\bot,a,b,\top\},
        \qquad
        \bot<a<\top,
        \quad
        \bot<b<\top,
\]
with \(a\) and \(b\) incomparable.  The poset \(\Dia\) is a finite lattice, hence a finite dcpo; its identity map is a finite-image deflation, so \(\Dia\) is an RB-domain.  Its Hasse graph is the four-cycle
\[
        \bot-a-\top-b-\bot.
\]
By Theorem~\ref{thm:cycle-obstruction}, \(\Vone(\Dia)\) is not an RB-domain.  Therefore \(\Vone\) does not preserve RB-domains.
\end{proof}

\begin{remark}
A detailed worked-out calculation for the diamond example \(\Dia\) will be given.  The calculation shows that the local order cone \(K_\Dia\) is non-simplicial, and that the finite-valued monotone approximants \(Q_n\) would have to satisfy a system of linear inequalities with no solution.  The same calculation applies to any finite poset with a Hasse cycle.
\end{remark}

\begin{remark}[Subprobabilistic and extended powerdomains]
\normalfont
The same method can be applied to the subprobabilistic powerdomains $\Vsub(P)$ and extended probabilistic
powerdomains $\Vext(P)$.  For background on extended probabilistic powerdomains, see
\cite{TixKeimelPlotkin2009}.  We only give the coordinate changes and the
resulting sketch, since the analytic obstruction is exactly the one proved in
Section~\ref{sec:euclidean}.

For the subprobabilistic powerdomain, write a subprobability valuation as its
point-mass vector
\[
        \iota_{\leq1}(\mu)=(\mu(x))_{x\in P}
        \in \R^P_{\geq0},
        \qquad
        \sum_{x\in P}\mu(x)\leq1.
\]
Under this point-coordinate embedding, the stochastic order is
\[
        \mu\leq \nu
        \quad\Longleftrightarrow\quad
        \sum_{x\in U}\mu(x)\leq \sum_{x\in U}\nu(x)
        \quad\text{for every upper set }U\subseteq P.
\]
Thus, at a finite interior point, the local order cone in point coordinates is
\[
        K^{\leq1}_P
        =\{h\in\R^P:h(U)\geq0
          \text{ for every upper set }U\subseteq P\},
        \qquad h(U)=\sum_{x\in U}h_x .
\]

The extended probabilistic powerdomain is treated in the same point
coordinates on its finite part:
\[
        \iota_{\mathrm{ext}}(\nu)=(\nu(x))_{x\in P}
        \in [0,\infty]^P,
\]
and locally at finite strictly positive points this again gives the cone
\[
        K^{\mathrm{ext}}_P
        =\{h\in\R^P:h(U)\geq0
          \text{ for every upper set }U\subseteq P\}.
\]
Hence the subprobabilistic and extended cases have the same finite-dimensional
local cone; the difference from the normalized case is that the tangent space is
\(\R^P\), not the hyperplane \(H_P=\{h:h(P)=0\}\).

For a connected component \(C\) of \(P\), let \(C_*\) be obtained by adjoining a
fresh least element \(*\).  The point-coordinate cone above is identified with
the normalized Hasse-flow cone of \(C_*\) by the linear map
\[
        h\longmapsto (-h(C),h)\in H_{C_*}\subseteq\R^{C_*}.
\]
Indeed, the inequality for an upper set \(U\subseteq C\) becomes the upper-set
inequality for the same upper set in \(C_*\), while the total coordinate sum in
\(C_*\) is zero.  Thus the local cone is simplicial exactly in the rooted-tree
case.  If \(C\) is not a rooted tree, then \(C_*\) has a Hasse cycle, and the
same non-simplicial-cone obstruction rules out the finite-valued monotone local
approximants forced by the RB property.

Conversely, when every connected component is a rooted tree, the positive
finite-tree argument applies componentwise.  The subprobabilistic and extended
powerdomains are then obtained from products of the corresponding tree
components, together with the standard finite-product and retract arguments.
We do not use this in the present paper.  The same proof scheme indicates the following classification, whose full details can be supplied by carrying out the componentwise version of the arguments above:
\[
\begin{aligned}
\Vsub(P)\text{ is an RB-domain}
&\Longleftrightarrow \Vext(P)\text{ is an RB-domain}\\
&\Longleftrightarrow \text{every connected component of }P
   \text{ is a rooted tree}.
\end{aligned}
\]
Here a rooted tree means a connected finite poset with a least element whose
principal ideals are chains.
\end{remark}

\section*{Acknowledgements}

The authors used LLM-based research tools during the preparation of this manuscript, including for exploring the Hasse-cycle obstruction and the finite-dimensional analytic argument.  
The authors are responsible for all mathematical statements and conclusions.

\end{document}